\documentclass[12pt]{amsart}
\usepackage{amsmath,amsthm,amssymb}
\newcommand{\nexteqv}{\displaybreak[0]\\ &\iff}

\newcommand{\nexteq}{\displaybreak[0]\\ &=}

\newcommand{\nextgeq}{\displaybreak[0]\\ &\geq}
\newcommand{\nextgq}{\displaybreak[0]\\ &>}
\newcommand{\nextlq}{\displaybreak[0]\\ &<}
\newcommand{\nextleq}{\displaybreak[0]\\ &\leq}

\newtheorem{lem}{Lemma}

\newtheorem{thm}[lem]{Theorem}
\newtheorem{prop}[lem]{Proposition}
\theoremstyle{definition}
\newtheorem{nota}[lem]{Notation}
\newtheorem{dfn}[lem]{Definition}

\newtheorem{exam}[lem]{Example}
\newtheorem{alg}{Algorithm}
\newtheorem{rmk}[lem]{Remark}

\DeclareMathOperator{\STab}{STab}

\begin{document}
\title[semistandard tableaux of given shape and weight]{On the greatest and least elements in the set of
semistandard tableaux of given shape and weight}

\author{Akihiro Munemasa}
\address{Research Center for Pure and Applied Mathematics, 
Graduate School of Information Sciences, 
Tohoku University, Sendai 980--8579, Japan}
\email{munemasa@math.is.tohoku.ac.jp}

\author{Minwon Na}
\address{Research Center for Pure and Applied Mathematics, 
Graduate School of Information Sciences, 
Tohoku University, Sendai 980--8579, Japan}
\email{minwon@ims.is.tohoku.ac.jp}

\date{June 29, 2015}

\begin{abstract}
We give three algorithms to construct a semistandard tableau of
given shape and weight, where the weight is a composition which
is not necessarily a partition.  With respect to a natural partial
order on the set of semistandard tableaux, we show that the
set of semistandard tableaux of given shape and weight has a
unique greatest element and a unique least element. Two of
our algorithms give each of these elements.
\end{abstract}
\maketitle

\section{Introduction}
Kostka numbers give the number of semistandard tableaux of
given shape and weight, and they play a fundamental role
in representation theory of symmetric groups (see \cite{S}).
Much work have been done on the problem of computing Kostka numbers,
which is known to be \#P complete (see \cite{N}).
In this paper, we study the \emph{set}, rather than the number,
of semistandard tableaux of
given shape and weight, in the hope that our study could shed some light
on the computation of the cardinality of the set in question.
We do not assume the weight is a partition, rather, it is an arbitrary
composition.

Throughout this paper, $n$ will denote a positive integer. Let
$\mu=(\mu_1,\mu_2\ldots,\mu_k)$ and $a=(a_1,a_2,\ldots,a_h)$ be a
partition and composition of $n$, respectively. We say that $h$ is
the height $h(a)$ of $a$. We denote by $D_\mu$ the Young diagram of
$\mu$, and by $\STab(\mu,a)$ the set of all semistandard tableaux of
shape $\mu$ and weight $a$. More precisely,
\begin{align*}
D_\mu&=\{(i,j)\in\mathbb{Z}^2\mid 1\leq i\leq k,\;1\leq j\leq\mu_i\},\\
\STab(\mu,a)&=\{T\mid T:D_\mu\to\{1,\dots,h\},\\
&\qquad\quad T(i,j)\leq
T(i,j+1),\;T(i,j)<T(i+1,j),\;(|T^{-1}(\{i\})|)_{i\geq 1}=a\}.
\end{align*}

For compositions $a=(a_1,a_2,\ldots,a_h)$ and
$b=(b_1,b_2,\ldots,b_k)$ of $n$, we say $a$ dominates $b$, denoted
$a\unrhd b$, if $k\geq h$ and
\[
\sum_{i=1}^j a_i\geq\sum_{i=1}^j b_i
\]
for $j=1,2,\ldots,h$. The following is well known.

\begin{thm}[{\cite[p.~26, Exercise 2]{F}}]\label{thm.2}
Let $\mu$ and $\lambda$ be partitions of $n$. Then
$\STab(\mu,\lambda)$ is nonempty if and only if $\mu\unrhd\lambda$.
\end{thm}

Let $\lambda(a)$ denote the partition of $n$ associated with $a$,
that is, the partition obtained from $a$ by rearranging the parts of
$a$ in non-increasing order. Then one can strengthen
Theorem~\ref{thm.2} using \cite[Lemma 3.7.1]{C}, as follows:

\begin{thm}[{\cite[p.~50, Proposition 2]{F}}]\label{lem.15}
Let $\mu$ and $a$ be a partition and composition of $n$,
respectively. Then $\STab(\mu,a)$ is nonempty if and only if
$\mu\unrhd\lambda(a)$.
\end{thm}

This theorem is incorrectly stated in \cite[Lemma 3.7.3]{C}, where
$\mu\unrhd\lambda(a)$ is replaced by $\mu\unrhd a$. For example, let
$\mu=(5,3)\vdash 8$ and $a=(2,6)\vDash 8$. Then $\mu\unrhd a$, but
$\STab(\mu,a)=\emptyset$.

The purpose of this note is to give explicit algorithms
to produce an element of $\STab(\mu,a)$, thereby giving a direct
proofs of Theorem~\ref{lem.15}. We also introduce a natural partial
order on $\STab(\mu,a)$ and show that it has unique
greatest and least elements, by showing that the elements
produced by two of the three algorithms have the respective property.
Although the proof of
Theorem~\ref{lem.15} using \cite[Lemma 3.7.1]{C} or \cite[p.~50,
Proposition 2]{F} gives, in principle, a bijection between
$\STab(\mu,a)$ and $\STab(\mu,\lambda(a))$, it does not give an
efficient algorithm to describe an element of $\STab(\mu,a)$ unless
the permutation required to transform $a$ to $\lambda(a)$ is a
transposition. We show that, in our
Proposition~\ref{lem.4} below,
a direct approach for proving Theorem~\ref{lem.15} along the line
of \cite[Lemma~3.7.3]{C} can be justified, and we describe
an algorithm to produce an element of $\STab(\mu,a)$ in this way.

This paper is organized as follows. After giving preliminaries in
Section 2, we describe the procedure of constructing the
greatest and least elements of $\STab(\mu,a)$ in Section
\ref{sec:1} and \ref{sec:3}, respectively. In Section~\ref{sec:2}, we
justify the proof of \cite[Lemma~3.7.3]{C} in
Proposition~\ref{lem.4}, using the ideas from Section~\ref{sec:1}.

\section{Preliminaries}
A composition of $n$ is a sequence $a=(a_1,a_2,\dots,a_h)$ of
positive integers such that $\sum_{i=1}^h a_i=n$. We write $a\vDash
n$ if $a$ is a composition of $n$. A partition of $n$ is a
non-increasing sequence $\mu=(\mu_1,\mu_2,\dots,\mu_k)$ of positive
integers such that $\sum_{i=1}^k \mu_i=n$. We write $\mu\vdash n$ if
$\mu$ is a partition of $n$. By convention, we define $\mu_i=0$ if
$i$ exceeds the number of parts in a partition $\mu$.

For a composition $a=(a_1,a_2,\ldots,a_h)\vDash n$, we define
\[
a^{(i)}=(a_1,\ldots,a_{i-1},a_i-1,a_{i+1},\ldots,a_h)
\]
for each $1\leq i\leq h$. Set
\begin{align*}
a'&=(a_1,a_2,\dots,a_{h-1})\vDash n-a_h,\\
\tilde{a}&=
\begin{cases}
(a_1,a_2,\dots,a_h-1)&\text{ if $a_h\geq 2$},\\
(a_1,a_2,\dots,a_{h-1})&\text{ if $a_h=1$}.
\end{cases}
\end{align*}
Then $\tilde{a}\vDash n-1$. Let
\[
q(a)=\max\{i\mid 1\leq i\leq h,\;\lambda(a)_i=a_h\}.
\]
Then
\begin{equation}\label{eq:01j}
\lambda(a)_{q(a)}=a_h>\lambda(a)_{q(a)+1}.
\end{equation}
Set
\[
\tilde{\lambda}(a)=\lambda(a)^{(q(a))}\vdash n-1.
\]
For a partition $\mu=(\mu_1,\dots,\mu_k)\vdash n$, we define
\[
s(\mu,a)=\max\{i\mid 1\leq i\leq k,\;\mu_i\geq a_h\}.
\]
Clearly,
\begin{equation}\label{eq:01s}
\mu_{s(\mu,a)}\geq a_h>\mu_{s(\mu,a)+1}.
\end{equation}

For $\rho=(\rho_1,\rho_2,\ldots,\rho_h)\vdash m$ and
$\mu=(\mu_1,\mu_2,\ldots,\mu_k)\vdash n$, we write
\[
\rho\preceq\mu
\]
if $m\leq n$, $h\leq k$ and $\rho_i\leq\mu_i$ for all $i$ with
$1\leq i\leq h$. For such $\rho$ and $\mu$, we say that $\mu/\rho$ is a
\emph{skew shape}, and we denote
$D_\mu\setminus D_\rho$ by $D_{\mu/\rho}$.
We say that the skew shape $\mu/\rho$ is totally
disconnected if $\rho_i\geq\mu_{i+1}$ for all $i$ with $1\leq i\leq
h$. Set
\begin{align*}
\mathcal{B}(\mu,a)=\{\rho\vdash n-a_h\mid
\rho\unrhd\lambda(a'),\;\rho\preceq\mu,\;\mu/\rho:\text{totally
disconnected}\}.
\end{align*}

\begin{lem}\label{lem.0a}
For a composition $a\vDash n$,
$\lambda(\tilde{a})=\tilde{\lambda}(a)$.
\end{lem}
\begin{proof}
Immediate from the definition.
\end{proof}

\begin{lem}\label{lem.3a}
Let $p$ and $q$ be positive integers. Let $\mu\vdash n$ and
$\lambda\vdash n$ satisfy $\mu_p>\mu_{p+1}$,
$\lambda_q>\lambda_{q+1}$ and $\mu\unrhd\lambda$. Then the following
are equivalent.
\begin{enumerate}
\item $\mu^{(p)}\unrhd\lambda^{(q)}$,
\item either $p\geq q$, or  $p<q$ and
$\sum_{i=1}^j\mu_i>\sum_{i=1}^j\lambda_i$ for all $j$ with $p\leq
j<q$.
\end{enumerate}
\end{lem}
\begin{proof}
Observe
\begin{align*}
\sum_{i=1}^j\mu^{(p)}_i &=\begin{cases} \sum_{i=1}^j\mu_i
&\text{if $1\leq j<p$,}\\
\sum_{i=1}^j\mu_i-1 &\text{otherwise,}
\end{cases}
\\
\sum_{i=1}^j\lambda^{(q)}_i&=\begin{cases} \sum_{i=1}^j\lambda_i
&\text{if $1\leq j<q$,}\\
\sum_{i=1}^j\lambda_i-1 &\text{otherwise.}
\end{cases}
\end{align*}
Thus
\[
\sum_{i=1}^j\mu^{(p)}_i\geq\sum_{i=1}^j\lambda^{(q)}_i \iff
\begin{cases}
\sum_{i=1}^j\mu_i\geq\sum_{i=1}^j\lambda_i&\text{if $1\leq j<\min\{p,q\}$,}\\
\sum_{i=1}^j\mu_i\geq\sum_{i=1}^j\lambda_i-1&\text{if
$q\leq j<p$,}\\
\sum_{i=1}^j\mu_i-1\geq\sum_{i=1}^j\lambda_i&\text{if
$p\leq j<q$,}\\
\sum_{i=1}^j\mu_i-1\geq\sum_{i=1}^j\lambda_i-1&\text{if
$\max\{p,q\}\leq j$.}
\end{cases}
\]
Since $\mu\unrhd\lambda$, we have
\begin{align*}
\text{(i)}&\iff \sum_{i=1}^j\mu_i-1\geq\sum_{i=1}^j\lambda_i
\quad\text{if }p\leq j<q, \nexteqv \text{(ii).}
\end{align*}
\end{proof}

\begin{dfn}\label{def.1}
Let $\mu\vdash n$ and $a\vDash n$ with $\mu\unrhd \lambda(a)$. A box
of coordinate $(i,\mu_i)$ is \textit{removable for the pair
$(\mu,a)$} if $\mu_i>\mu_{i+1}$ and
$\mu^{(i)}\unrhd\lambda(\tilde{a})$. We denote by $R{(\mu,a)}$ the set
of all $i$ such that $(i,\mu_i)$ is removable for the pair
$(\mu,a)$.
\end{dfn}

\begin{lem}\label{lem.10}
Let $\mu\vdash n$ and $a\vDash n$ satisfy $\mu\unrhd\lambda(a)$.
Then $s(\mu,a)\in R{(\mu,a)}$.
\end{lem}
\begin{proof}
Write $p=s(\mu,a)$, $q=q(a)$, $\lambda=\lambda(a)$. Then
$\lambda_q>\lambda_{q+1}$ by (\ref{eq:01j}) and $\mu_p>\mu_{p+1}$ by
(\ref{eq:01s}). Moreover, by Lemma~\ref{lem.0a}, we have
$\lambda(\tilde{a})=\tilde{\lambda}(a)=\lambda^{(q)}$. Thus, in view
of Lemma~\ref{lem.3a}, it suffices to show
\[
\text{either }p\geq q, \text{ or }p<q\text{ and
}\sum_{i=1}^j\mu_i>\sum_{i=1}^j\lambda_i \text{ for all $j$ with
}p\leq j<q.
\]
Suppose $p<q$ and let $p\leq j<q$. Let $a=(a_1,\dots,a_h)$. If
$j+1\leq i\leq q$, then $p<i\leq q$, so $\mu_i<a_h\leq\lambda_i$.
Thus
\begin{align*}
\sum_{i=1}^j\mu_i &=\sum_{i=1}^q\mu_i-\sum_{i=j+1}^q\mu_i \nextgeq
\sum_{i=1}^q\lambda_i-\sum_{i=j+1}^q\mu_i \nextgq
\sum_{i=1}^q\lambda_i-\sum_{i=j+1}^q\lambda_i \nexteq
\sum_{i=1}^j\lambda_i.
\end{align*}
\end{proof}

From Lemma~\ref{lem.10}, we find $R{(\mu,a)}\neq\emptyset$. Set
\[l(\mu,a)=\min R{(\mu,a)}.\]

\begin{lem}\label{lem.5}
Let $\mu\vdash n$ and $a\vDash n$ satisfy $\mu\unrhd\lambda(a)$.
Then $l(\mu,a)\leq s(\mu,a)$.
\end{lem}
\begin{proof}
Immediate from Lemma~\ref{lem.10}.
\end{proof}

\begin{lem}\label{lem.4b}
Let $\mu=(\mu_1,\mu_2,\ldots,\mu_k)\vdash n$ and
$a=(a_1,a_2,\ldots,a_h)\vDash n$. Let $i$ be an integer with $1\leq
i\leq k$. If there exists a tableau $T\in\STab(\mu, a)$ such that
$T(i,\mu_i)=h$, then $i\in R{(\mu,a)}$.
\end{lem}

\begin{proof}
Since $T\in\STab(\mu,a)$ and $T(i,\mu_i)=h$, we have
$(i+1,\mu_i)\notin D_{\mu}$, so $\mu_{i}>\mu_{i+1}$. Also,
$\tilde{T}=T|_{D_{\mu}\setminus\{(i,\mu_i)\}}\in\STab(\mu^{(i)},\tilde{a})$.
By {\cite[Lemma 3.7.1]{C}}, we obtain
$\STab(\mu^{(i)},\lambda(\tilde{a}))\neq\emptyset$. Thus
$\mu^{(i)}\unrhd\lambda(\tilde{a})$ and $i\in R{(\mu,a)}$.
\end{proof}

In fact, the converse of Lemma~\ref{lem.4b} is also true. We will
prove it in Section~\ref{sec:2}.

\begin{lem}\label{lem.19}
Let $\mu\vdash n$ and $a=(a_1,a_2,\ldots,a_h)\vDash n$ satisfy
$\mu\unrhd\lambda(a)$. For $T\in\STab(\mu,a)$, we have
\[
l(\mu,a)\leq\min\{i\mid T(i,\mu_i)=h\}\leq s(\mu,a).
\]
\end{lem}

\begin{proof}
Write
$q=\min\{i\mid T(i,\mu_i)=h\}$.
By Lemma~\ref{lem.4b},
we have $l(\mu,a)\leq q$.
Since $a_h=|\{(i,j)\in D_\mu\mid T(i,j)=h\}|\leq\mu_q$, we have
$q\leq s(\mu,a)$.
\end{proof}

We will show in Theorem~\ref{thm.1a}, Lemma~\ref{lem.16} and
Theorem~\ref{thm.3} that equality can be achieved in both of the
inequalities above. In Section~\ref{sec:1}, we give an algorithm to
construct $T\in\STab(\mu,a)$ such that $\min\{i\mid
T(i,\mu_i)=h\}=s(\mu,a)$. In Sections~\ref{sec:2} and \ref{sec:3},
we give an algorithm to construct $T\in\STab(\mu,a)$ such that
$\min\{i\mid T(i,\mu_i)=h\}=l(\mu,a)$.

Finally, we define a partial order on $\STab(\mu,a)$ and a partition
$\rho(\mu,a)\vdash n-a_h$ as follows. We write $\mu\rhd\lambda$ to
mean $\mu\unrhd\lambda$ and $\mu\neq\lambda$.

\begin{dfn}
Let $\mu=(\mu_1,\mu_2,\ldots,\mu_k)\vdash n$ and
$a=(a_1,a_2,\ldots,a_h)\vDash n$ satisfy $\mu\unrhd\lambda(a)$. For
$T$, $S\in\STab(\mu,a)$, let
\begin{align}
\tau^{(p)}&=(|\{j\mid T(i,j)\leq p\}|)_{i=1}^k,\label{def.eq.1}\\
\sigma^{(p)}&=(|\{j\mid S(i,j)\leq p\}|)_{i=1}^k\label{def.eq.2}
\end{align}
for all $p$ with $1\leq p\leq h$. We define that $S\leq T$ if,
either $T=S$ or, $\tau^{(h)}=\sigma^{(h)}$,
$\tau^{(h-1)}=\sigma^{(h-1)},\ldots,\tau^{(p+1)}=\sigma^{(p+1)}$,
$\tau^{(p)}\rhd\sigma^{(p)}$
for some $1\leq p\leq h$.
\end{dfn}

Since the relation $\unrhd$ is a partial order, we see that
$(\STab(\mu,a),\leq)$ is a partially ordered set.

Alternatively the partial order $\leq$ on $\STab(\mu,a)$ can be
defined recursively as follows: for $T,S\in\STab(\mu,a)$, define
$\tau$ and $\sigma$ by
\begin{align}
T^{-1}(\{1,\dots,h-1\})&=D_{\tau},\label{Tt}\\
S^{-1}(\{1,\dots,h-1\})&=D_{\sigma}\label{Ss},
\end{align}
respectively. We define $S\leq T$ if, either $\tau\rhd\sigma$, or
$\tau=\sigma$ and $S|_{D_\sigma}\leq T|_{D_\sigma}$.

\begin{dfn}
Let $\mu=(\mu_1,\ldots,\mu_k)\vdash n$ and $a=(a_1,\dots,a_h)\vDash
n$ satisfy $\mu\unrhd\lambda(a)$. Define
$\rho(\mu,a)=(\rho_1,\dots,\rho_{k-1})\vdash n-a_h$ by setting
\begin{align*}
\rho_i&=\begin{cases} \mu_i&\text{ if } 1\leq i<s,\\
\mu_s-(a_h-\mu_{s+1})&\text{ if }i=s,\\
\mu_{i+1}&\text{ if }s<i\leq k-1,
\end{cases}
\end{align*}
where $s=s(\mu,a)$.
\end{dfn}

\section{The greatest element of $\STab(\mu,a)$}\label{sec:1}

\begin{lem}\label{lem.8b}
Let $\mu=(\mu_1,\ldots,\mu_k)\vdash n$ and $a=(a_1,\dots,a_h)\vDash
n$ satisfy $\mu\unrhd\lambda(a)$. Then $\rho(\mu,a)$ is the greatest
element of $\mathcal{B}(\mu,a)$.
\end{lem}

\begin{proof}
Write $\rho=\rho(\mu,a)$ and $s=s(\mu,a)$. By (\ref{eq:01s}), we
have $\mu_s\geq a_h>\mu_{s+1}$. Thus $\rho\vdash n-a_h$ and
$\mu_s>\rho_s\geq\mu_{s+1}$. So $\mu_i\geq\rho_i\geq\mu_{i+1}$ for
all $i$ with $1\leq i\leq k$. This implies that $\rho\preceq\mu$ and
$\mu/\rho$ is totally disconnected.

Next, we show that $\rho\unrhd\lambda(a')$. Write
$\lambda(a)=(\lambda_1,\dots,\lambda_h)$,
$\lambda'=(\lambda_1',\dots,\lambda_{h-1}')=\lambda(a')$ and
$q=q(a)$. Then $\lambda_q=a_h$ and
\[
\lambda'=(\lambda_1,\lambda_2,\ldots,\lambda_{q-1},\lambda_{q+1},\lambda_{q+2},\ldots,\lambda_h)\vdash
n-\lambda_q.
\]
Observe
\begin{align}
\sum_{i=1}^j\rho_i&=
\begin{cases}
\sum_{i=1}^j\mu_i&\text{if $1\leq j<s$,}\\
\sum_{i=1}^{j+1}\mu_i-\lambda_q&\text{if $s\leq j\leq k-1$}
\end{cases}
\notag\nextgeq
\begin{cases}
\sum_{i=1}^j\lambda_i&\text{if $1\leq j<s$,}\\
\sum_{i=1}^{j+1}\lambda_i-\lambda_q&\text{if $s\leq j\leq k-1$}
\end{cases}
\label{8b.eq.0}
\end{align}
since $\mu\unrhd\lambda(a)$.

\textbf{Case 1.} $q<s$. By (\ref{8b.eq.0}), we have
\begin{align*}
\sum_{i=1}^j\rho_i&\geq
\begin{cases}
\sum_{i=1}^j\lambda_i&\text{if $1\leq j<q$,}\\
\sum_{i=1}^{q-1}\lambda_i+\sum_{i=q+1}^{j+1}\lambda_i
+\lambda_q-\lambda_{j+1}&\text{if $q\leq j<s$,}\\
\sum_{i=1}^{j+1}\lambda_i-\lambda_q&\text{if $s\leq j\leq k-1$}
\end{cases}
\nextgeq
\begin{cases}
\sum_{i=1}^j\lambda_i&\text{if $1\leq j<q$,}\\
\sum_{i=1}^{q-1}\lambda_i+\sum_{i=q+1}^{j+1}\lambda_i&\text{if $q\leq j<s$,}\\
\sum_{i=1}^{j+1}\lambda_i-\lambda_q&\text{if $s\leq j\leq k-1$}
\end{cases}
\nexteq \sum_{i=1}^{j}\lambda'_i.
\end{align*}

\textbf{Case 2.} $s\leq q$. By (\ref{8b.eq.0}), we have
\begin{align*}
\sum_{i=1}^j\rho_i&\geq
\begin{cases}
\sum_{i=1}^j\lambda_i&\text{if $1\leq j<s$,}\\
\sum_{i=1}^{j}\lambda_i+\lambda_{j+1}-\lambda_q&\text{if $s\leq j<\min\{q,k\}$,}\\
\sum_{i=1}^{j+1}\lambda_i-\lambda_q&\text{if $\min\{q,k\}\leq j\leq
k$}
\end{cases}
\nextgeq
\begin{cases}
\sum_{i=1}^{j}\lambda_i&\text{if $1\leq j<\min\{q,k\}$,}\\
\sum_{i=1}^{j+1}\lambda_i-\lambda_q&\text{if $q\leq j\leq k$}
\end{cases}
\nextgeq \sum_{i=1}^j\lambda'_i.
\end{align*}
Thus we have $\rho\in\mathcal{B}(\mu,a)$.

It remains to show that $\rho\unrhd\tau$ for all
$\tau\in\mathcal{B}(\mu,a)$. Let $\tau\in\mathcal{B}(\mu,a)$ with
$\tau\neq\rho$ and
\begin{equation}\label{8c.eq.5}
r=\min\{i\mid 1\leq i\leq k,\;\tau_i<\mu_i\}.
\end{equation}
Then
\begin{equation}\label{8c.eq.4}
1\leq r\leq s
\end{equation}
and
\begin{equation}\label{8c.eq.2}
\sum_{i=r}^k(\mu_i-\tau_i)= a_h.
\end{equation}
By the definition of $\mathcal{B}(\mu,a)$, we have
\begin{align}\label{8c.eq.1}
\mu_{i}\geq \tau_i\geq\mu_{i+1}
\end{align}
for all $i$ with $1\leq i\leq k$.

If $1\leq j<s$, then
\[
\sum_{i=1}^j\rho_i=\sum_{i=1}^j\mu_i\geq\sum_{i=1}^j\tau_i.
\]

If $s\leq j\leq k$, then
\begin{align*}
\sum_{i=1}^j\rho_i&=\sum_{i=1}^{j+1}\mu_i-a_h\nexteq
\sum_{i=1}^{j+1}\mu_i-\sum_{i=r}^k(\mu_i-\tau_i)&&\text{(by
(\ref{8c.eq.2}))} \nexteq
\sum_{i=1}^{r-1}\mu_i+\sum_{i=r}^{j+1}\mu_i-\sum_{i=r}^k\mu_i+\sum_{i=r}^k\tau_i&&\text{(by
(\ref{8c.eq.4}))} \nexteq
\sum_{i=1}^{j}\tau_i+\sum_{i=j+1}^k\tau_i-\sum_{i=j+2}^k\mu_i
&&\text{(by (\ref{8c.eq.5}))} \nexteq
\sum_{i=1}^{j}\tau_i+\sum_{i=j+1}^k(\tau_i-\mu_{i+1})\nextgeq
\sum_{i=1}^{j}\tau_i&&\text{(by (\ref{8c.eq.1})).}
\end{align*}
Therefore, $\rho\unrhd\tau$.
\end{proof}

\begin{thm}\label{thm.1a}
Given $\mu\vdash n$ and $a=(a_1,a_2,\ldots,a_h)\vDash n$ such that
$\mu\unrhd\lambda(a)$, define $\rho^i$ and $a^i$ inductively by
setting $\rho^0=\mu$, $a^0=a$, and for $1\leq i\leq h$,
\begin{align*}
\rho^{i}&=\rho(\rho^{i-1},a^{i-1})\vdash \sum_{j=1}^{h-i}a_j,\\
a^{i}&=(a^{i-1})'\vDash \sum_{j=1}^{h-i}a_j.
\end{align*}
Define a tableau $T$ of shape $\mu$ and weight $a$ by
\begin{equation}\label{t1.eq.1a}
T(p,q)=h-i\text{ if }(p,q)\in D_{\rho^i/\rho^{i+1}}.
\end{equation}
Then $T$
is the greatest element of $\STab(\mu,a)$.
\end{thm}

\begin{proof}
We prove the assertion by induction on $h$. Suppose first $h=1$.
Then $\STab(\mu,a)$ consists of a single element $T$, so that the
assertion trivially holds.

Next suppose $h>1$. Assume that the assertion holds for $h-1$.
Set $\nu=\rho^1$ and $b=a^1$.
Since $\rho^1\in\mathcal{B}(\mu,a)$ by Lemma~\ref{lem.8b},
we have $\nu\unrhd\lambda(b)$. Define $\nu^i$ and $b^i$ inductively by
setting $\nu^0=\nu$, $b^0=b$, and for $1\leq i<h$,
\begin{align*}
\nu^{i}&=\rho(\nu^{i-1},b^{i-1})\vdash \sum_{j=1}^{h-1-i}b_j,\\
b^{i}&=(b^{i-1})'\vDash \sum_{j=1}^{h-1-i}b_j.
\end{align*}
Define a tableau $T'$ of shape $\nu$ and weight $b$ by
\[
T'(p,q)=h-(i+1)\text{ if }(p,q)\in D_{\nu^{i}/\nu^{i+1}}.
\]
By the inductive hypothesis,
\begin{align}
T'&\in\STab(\nu,b),\label{t1.eq.4}\\
T'&\geq S'\text{ for all }S'\in\STab(\nu,b).\label{t1.eq.3}
\end{align}
It is easy to show that $b^i=a^{i+1}$ and $\nu^i=\rho^{i+1}$ by induction on $i$,
and the latter implies $T|_{D_{\nu}}=T'$.
Then by (\ref{t1.eq.4}) and the fact that
$\mu/\nu$ is totally disconnected, we obtain
$T\in\STab(\mu,a)$.

It remains to show that $T\geq S$ for all $S\in\STab(\mu,a)$.
Let $S\in\STab(\mu,a)$.
Define a partition $\sigma$ by (\ref{Ss}).
By Lemma~\ref{lem.8b}, we have $\nu\unrhd\sigma$.
If $\nu\rhd\sigma$, then $T\geq S$. If $\nu=\sigma$,
then $T|_{D_\nu}=T'\geq S|_{D_\nu}$ by (\ref{t1.eq.3}), hence $T\geq S$.
\end{proof}

From Theorem~\ref{thm.1a}, we obtain a tableau $T\in\STab(\mu,a)$.

\begin{alg}\label{a.2}
\textbf{Input}: $\mu\vdash n$ and
$a\vDash n$ such that $\mu\unrhd\lambda(a)$.\\
\textbf{Output}: $T\in\STab(\mu,a)$.\\
\textbf{Initialization}: $\nu:=\mu$, $b:=a$.\\
\textbf{while} $h(b)>1$ \textbf{do}\\
\begin{tabular}{c|l}
&\;$T(i,j):=h(b)$ \textbf{where} $(i,j)\in D_{\nu/\rho(\nu,b)}$.\\
&\;$\nu\leftarrow\rho(\nu,b)$, $b\leftarrow b'$.\\
\end{tabular}\\
\textbf{end}\\
$T(1,j):=1$ \textbf{where} $1\leq j\leq \nu_1$.\\
Output $T$.
\end{alg}

\begin{exam}
Let $\mu=(4,4,1,1)\vdash 10$ and $a=(1,3,2,2,2)\vDash 10$. Then a
tableau $T\in\STab(\mu,a)$ is obtained via Algorithm~\ref{a.2}.
\[
\begin{tabular}{|c|c|c|c|}
\hline
$\nu$ & $b$ & $\rho(\nu,b)$ & $T$  \\
\hline $(4,4,1,1)$ & $(1,3,2,2,2)$ & $(4,3,1)$ & $T(2,4)=T(4,1)=5$ \\
\hline $(4,3,1)$& $(1,3,2,2)$ & $(4,2)$ & $T(2,3)=T(3,1)=4$ \\
\hline $(4,2)$& $(1,3,2)$ & $(4)$ & $T(2,1)=T(2,2)=3$ \\
\hline $(4)$& $(1,3)$ & $(1)$ & $T(1,2)=T(1,3)=T(1,4)=2$ \\
\hline $(1)$& $(1)$ &  & $T(1,1)=1$ \\
\hline
\end{tabular}
\]
Thus
\[
T=\begin{array}{cccc}
1&2&2&2\\
3&3&4&5\\
4&&&\\
5
\end{array}
\]
is the greatest element of $\STab(\mu,a)$.
\end{exam}

\section{Removable boxes}\label{sec:2}

Throughout this section, let $\mu=(\mu_1,\mu_2,\ldots,\mu_k)\vdash
n$ and $a=(a_1,a_2,\ldots,a_h)\vDash n$. Let
$\lambda=\lambda(a)=(\lambda_1,\lambda_2,\ldots,\lambda_h)$,
$q=q(a)$ and $l=l(\mu,a)$. We assume $\mu\unrhd\lambda$.

\begin{lem}\label{lem.8a}
Assume $a_h\geq 2$. For $i\in R{(\mu,a)}$ with $i\leq
s(\mu^{(i)},\tilde{a})$, we have
$\rho(\mu^{(i)},\tilde{a})\in\mathcal{B}(\mu,a)$.
\end{lem}

\begin{proof}
Write $\rho=\rho(\mu^{(i)},\tilde{a})$. Since $i\in R(\mu,a)$, we
have $\mu^{(i)}\unrhd\lambda(\tilde{a})$. Then by
Lemma~\ref{lem.8b}, $\rho\in\mathcal{B}(\mu^{(i)},\tilde{a})$. Since
$a_h\geq2$, this implies $\rho\unrhd\lambda(\tilde{a}')
=\lambda(a')$. To prove $\rho\in\mathcal{B}(\mu,a)$, it remains to
show that $\mu/\rho$ is totally disconnected. Since $\mu^{(i)}/\rho$
is totally disconnected, it is enough to show $\rho_{i-1}\geq\mu_i$.
Since $i-1<s(\mu^{(i)},\tilde{a})$, we obtain
$\rho_{i-1}=\mu^{(i)}_{i-1}=\mu_{i-1}\geq\mu_i$.
\end{proof}

\begin{lem}\label{lem.4a}
Assume $a_h\geq 2$. Then $r\leq s(\mu,a)$ if and only if $r\leq s(\mu^{(r)},\tilde{a})$.
\end{lem}

\begin{proof}
Immediate from the definition.
\end{proof}

\begin{lem}\label{lem.16}
Define $\rho\vdash n-a_h$ by
\begin{align*}
\rho=\begin{cases}\mu^{(l)}&\text{ if $a_h=1$,}\\
\rho(\mu^{(l)},\tilde{a})&\text{ if $a_h\geq2$}.
\end{cases}
\end{align*}
Then $\mu/\rho$ is totally disconnected, $\rho_l<\mu_l$ and
$\rho\unrhd\lambda(a')$.
\end{lem}

\begin{proof}
If $a_h=1$, then
$\rho=\mu^{(l)}\unrhd\lambda(\tilde{a})=\lambda(a')$, since $l\in
R(\mu,a)$.
Thus the assertion holds.

Suppose $a_h\geq2$.
Then $l\leq s(\mu^{(l)},\tilde{a})$ by Lemma~\ref{lem.5}, Lemma~\ref{lem.4a}, and hence
$\rho\in\mathcal{B}(\mu,a)$ by Lemma~\ref{lem.8a}.
Thus it remains to show that $\rho_l<\mu_l$. This can be shown as
follows:
\begin{align*}
\rho_l&=
\begin{cases}
\mu^{(l)}_l-(\tilde{a}_h-\mu^{(l)}_{l+1})&\text{ if $l=s(\mu^{(l)},\tilde{a})$,}\\
\mu^{(l)}_l&\text{ if $l<s(\mu^{(l)},\tilde{a})$}
\end{cases}
\nextleq \mu^{(l)}_l
\nextlq\mu_l,
\end{align*}
where the second inequality follows from the definition of
$s(\mu^{(l)},\tilde{a})$.
\end{proof}

From Lemma~\ref{lem.16}, we obtain a tableau $U\in\STab(\mu,a)$.

\begin{alg}\label{a.3}
\textbf{Input}: $\mu=(\mu_1,\mu_2,\ldots,\mu_k)\vdash n$ and
$a=(a_1,a_2,\ldots,a_h)\vDash n$ such that $\mu\unrhd\lambda(a)$.\\
\textbf{Output}: $U\in\STab(\mu,a)$.\\
\textbf{Initialization}: $\nu:=\mu$, $b:=a$.\\
\textbf{while} $h(b)>1$ \textbf{do}\\
\begin{tabular}{c|l}
&\;$l:=l(\nu,b)$.\\
&\;\textbf{if} $b_{h(b)}=1$ \textbf{then} $\rho:=\nu^{(l)}$.\\
&\;\textbf{else}\\
&\;\;$\rho:=\rho(\nu^{(l)},\tilde{b})$.\\
&\;\textbf{end if}\\
&\;\;$U(i,j):=h(b)$ \textbf{where} $(i,j)\in D_{\nu/\rho}$.\\
&\;$\nu\leftarrow\rho$, $b\leftarrow b'$.\\
\end{tabular}\\
\textbf{end}\\
$U(1,j):=1$ \textbf{where} $1\leq j\leq\nu_1$.\\
Output $U$.
\end{alg}

\begin{exam}
Let $\mu=(4,4,1,1)\vdash 10$ and $a=(1,3,2,2,2)\vDash 10$. Then a
tableau $U\in\STab(\mu,a)$ is obtained via Algorithm \ref{a.3}.
\[
\begin{tabular}{|c|c|c|c|}
\hline
$\nu$ & $b$ & $\rho$ & $U$  \\
\hline $(4,4,1,1)$ & $(1,3,2,2,2)$ & $(4,3,1)$ & $U(2,4)=U(4,1)=5$ \\
\hline $(4,3,1)$ & $(1,3,2,2)$ & $(3,3)$ & $U(1,4)=U(3,1)=4$ \\
\hline $(3,3)$ & $(1,3,2)$ & $(3,1)$ & $U(2,2)=U(2,3)=3$ \\
\hline $(3,1)$ & $(1,3)$ & $(1)$ & $U(1,2)=U(1,3)=U(2,1)=2$ \\
\hline $(1)$ & $(1)$ &  & $U(1,1)=1$ \\
\hline
\end{tabular}
\]
Thus
\[
U=\begin{array}{cccc}
1&2&2&4\\
2&3&3&5\\
4&&&\\
5
\end{array}.
\]
We note that the least element of $\STab(\mu,a)$ is
\[
S=\begin{array}{cccc}
1&2&2&4\\
2&3&5&5\\
3&&&\\
4
\end{array}.
\]
In Section~\ref{sec:3}, we will show that there exists a unique
least element of $\STab(\mu,a)$ whenever $\mu\unrhd\lambda(a)$, and
give an algorithm to construct it.
\end{exam}

\begin{prop}\label{lem.4}
Let $\mu=(\mu_1,\mu_2,\ldots,\mu_k)\vdash n$ and
$a=(a_1,a_2,\ldots,a_h)\vDash n$, and assume $\mu\unrhd\lambda(a)$.
Let $r$ be an integer with $1\leq r\leq k$. Then there exists a
tableau $T\in\STab(\mu,a)$ such that $T(r,\mu_r)=h$ if and only if
$r\in R{(\mu,a)}$.
\end{prop}

\begin{proof}
The ``only if'' part has been proved in Lemma~\ref{lem.4b}. We prove
the ``if'' part by induction on $n$. If $n=1$ then it is obvious.
Let $r\in R{(\mu,a)}$, $s=s(\mu,a)$ and $s'=s(\mu^{(r)},\tilde{a})$.

If $a_h\geq 2$ then define $\rho\vdash n-a_h$ by
\begin{align}\label{4.eq.1}
\rho=\begin{cases} \rho(\mu,a)&\text{ if $r>s'$,}\\
\rho(\mu^{(r)},\tilde{a})&\text{otherwise}.
\end{cases}
\end{align}
From Lemma~\ref{lem.8b} and $\ref{lem.8a}$, we have
$\rho\in\mathcal{B}(\mu,a)$, so
\begin{equation}\label{eq:thm20}
\rho\unrhd\lambda(a').
\end{equation}

If $a_h=1$, then define by $\rho=\mu^{(r)}$. From the definition of
$R{(\mu,a)}$, (\ref{eq:thm20}) holds in this case also.

Since (\ref{eq:thm20}) implies $R(\rho,\lambda(a'))\neq\emptyset$,
the inductive hypothesis implies that there exists a tableau
$T'\in\STab(\rho,a')$. Define a tableau $T$ of shape $\mu$ and
weight $a$ by
\[T(i,j)=
\begin{cases}
T'(i,j) &\text{ if }(i,j)\in D_{\rho},\\
h &\text{ if } (i,j)\in D_{\mu/\rho}.
\end{cases}
\]

It remains to show that $T(r,\mu_r)=h$. This will follow if we can
show $\rho_r<\mu_r$. If $a_h=1$, then $\rho_r=\mu_r-1<\mu_r$.
Suppose $a_h\geq 2$. If $r>s'$ then we have $r>s'\geq s$ by the
definition of $s'$ and $s$. Since $r\in R{(\mu,a)}$, we have
$\rho_r=\mu_{r+1}<\mu_r$. If $r\leq s'$, then
\begin{align*}
\rho_r&=
\begin{cases}
\mu^{(r)}_r-(\tilde{a}-\mu^{(r)}_{r+1})&\text{ if $r=s'$,}\\
\mu^{(r)}_r&\text{ if $r<s'$}
\end{cases}
\nextleq \mu^{(r)}_r\nexteq \mu_r-1\nextlq\mu_r,
\end{align*}
where the second inequality follows from the definition of $s'$.
\end{proof}

Proposition~\ref{lem.4} justifies the proof of \cite[Lemma
3.7.3]{C}. It also gives an alternative proof of the ``if'' part of
Theorem~\ref{lem.15}.

\section{The least element of $\STab(\mu,a)$}\label{sec:3}

Throughout this section, we let $\mu=(\mu_1,\mu_2,\ldots,\mu_k)\vdash
n$ and $a=(a_1,a_2,\ldots,a_h)\vDash n$. We assume
$\mu\unrhd\lambda(a)$. For a sequence $(i_1,i_2,\ldots,i_j)$ of positive integers,
we abbreviate the partition
\[(\cdots((\mu^{(i_1)})^{(i_2)})\cdots)^{(i_j)}\]
of $n-j$, as $\mu^{(i_1,i_2,\ldots,i_j)}$.

\begin{lem}\label{lem.3b}
We have
\[
R{(\mu,a)}=\{i\mid l(\mu,a)\leq i\leq k,\;\mu_{i}>\mu_{i+1}\}.
\]
\end{lem}

\begin{proof}
Write $q=q(a)$, $\lambda=\lambda(a)$ and $l=l(\mu,a)$. Then
$\lambda_q>\lambda_{q+1}$ by (\ref{eq:01j}). Let $l\leq i\leq k$ and
$\mu_{i}>\mu_{i+1}$. From Lemma~\ref{lem.3a} and the definition of
$l$, we have $\mu^{(i)}\unrhd\mu^{(l)}\unrhd\lambda^{(q)}$. Thus
$i\in R{(\mu,a)}$.
\end{proof}

For each $i$ with $1\leq i\leq k$, set
\[
R{(\mu,a,i)}=\{r\in R{(\mu,a)}\mid r\geq i\}.
\]
From Lemma~\ref{lem.3b}, we have $k\in R{(\mu,a,i)}$ for each $i$
with $1\leq i\leq k$. Set
\[
l(\mu,a,i)=\min R{(\mu,a,i)}.
\]
Clearly, $l(\mu,a,1)=l(\mu,a)$.

\begin{lem}\label{lem.3d}
Let $\mu$ and $\mu'$ be partitions of $n$.
Suppose that $i\in R(\mu,a)$ and $i'\in R(\mu',a)$ satisfy $i\leq i'$,
$\mu_{i'}\geq \mu'_{i'}$ and $\mu_j=\mu'_j$ for all $j$ with $j>i'$. Then
\[
R(\mu'^{(i')},\tilde{a},i')\subseteq
R(\mu^{(i)},\tilde{a},i).
\]
\end{lem}
\begin{proof}
Let $r\in R(\mu'^{(i')},\tilde{a},i')$.
Since $i\leq i'\leq r$, we have
\begin{equation}\label{eq.3d.5}
\delta_{i,r}\leq\delta_{i',r},
\end{equation}
and
\begin{equation}\label{eq.3d.4}
\mu'^{(i')}_{r+1}=\mu'_{r+1}=\mu_{r+1}=\mu^{(i)}_{r+1}.
\end{equation}
Thus
\begin{align*}
\mu^{(i)}_r&=\mu_r-\delta_{i,r}
\\&\geq
\mu'_r-\delta_{i,r}
\\&\geq
\mu'_r-\delta_{i',r}
&&\text{(by (\ref{eq.3d.5}))}
\nexteq
\mu'^{(i')}_r
\nextgq
\mu'^{(i')}_{r+1}
\nexteq\mu^{(i)}_{r+1}
&&\text{(by (\ref{eq.3d.4}))}.
\end{align*}

It remains show that
$\mu^{(i,r)}\unrhd\lambda(\tilde{\tilde{a}})$. If $1\leq q<r$, then
\[
\sum_{j=1}^q\mu^{(i,r)}_j=\sum_{j=1}^q\mu^{(i)}_j\geq
\sum_{j=1}^q\lambda(\tilde{a})_j
\geq\sum_{j=1}^q\lambda(\tilde{\tilde{a}})_j,
\]
since $i\in R(\mu,a)$. If $r\leq q$, then $i'\leq q$, and hence
$\sum_{j>q}\mu_j=\sum_{j>q}\mu'_j$. Thus
\begin{align*}
\sum_{j=1}^q\mu^{(i,r)}_j&=\sum_{j=1}^q\mu_j-2\nexteq
\sum_{j=1}^{q}\mu'_j-2\nexteq
\sum_{j=1}^{q}\mu'^{(i',r)}_j\nextgeq
\sum_{j=1}^q\lambda(\tilde{\tilde{a}})_j,
\end{align*}
by $r\in R(\mu'^{(i')},\tilde{a})$.
\end{proof}

\begin{lem}\label{lem.6}
Assume $a_h\geq 2$, $r\in R(\mu,a)$, and $r\leq s(\mu,a)$.
Then
$s(\mu,a)\leq s(\mu^{(r)},\tilde{a})$
In particular, $R(\mu^{(r)},\tilde{a},r)\neq\emptyset$
and $l(\mu^{(r)},\tilde{a},r)\leq s(\mu^{(r)},\tilde{a})$.
\end{lem}
\begin{proof}
Since $a_h\geq2$, we have
\begin{align*}
\tilde{a}_h&=a_h-1
\\&\leq
a_h-\delta_{r,s(\mu,a)}
\nextleq
\mu_{s(\mu,a)}-\delta_{r,s(\mu,a)}
\nexteq
\mu^{(r)}_{s(\mu,a)}.
\end{align*}
Thus $s(\mu^{(r)},\tilde{a})\geq s(\mu,a)\geq r$,
and hence $s(\mu^{(r)},\tilde{a})\in R(\mu^{(r)},\tilde{a},r)$
by Lemma~\ref{lem.10}.
\end{proof}

\begin{nota}\label{not.3}
Let $r\in R(\mu,a)$ and suppose $r\leq s(\mu,a)$.
Define $a^i$, $l_i$ and $\mu^i$
inductively by setting  $a^0=a$, $l_0=r$, $\mu^0=\mu$
and for $0\leq i< n$,
\begin{align*}
a^{i+1}&=\widetilde{a^{i}}\vDash n-i-1,\\
l_{i+1}&=
\begin{cases}
l(\mu^{i},a^{i},1)&\text{ if }i\in A,\\
l(\mu^{i},a^{i},l_{i})&\text{ otherwise,}
\end{cases}\\
\mu^{i+1}&=(\mu^{i})^{(l_{i+1})}\vdash n-i-1,
\end{align*}
where $A=\{a_h,a_h+a_{h-1},\ldots,a_h+\cdots+a_2\}$.

In order to check $l_{i+1}$ and $\mu^{i+1}$ are well-defined,
we show
\begin{align}
\mu^{i}&\unrhd\lambda(a^i)&&(0\leq i< n),\label{17.eq.7}\\
R(\mu^{i},a^{i},l_{i})&\neq\emptyset&&(0\leq i<n,\;i\notin A),
\label{w1}\\
l_{i+1}&\leq s(\mu^i,a^i)&&(0\leq i< n).\label{w0}
\end{align}
Indeed, (\ref{17.eq.7})--(\ref{w1}) guarantee that
$l_{i+1}$ is defined as an element of $R(\mu^i,a^i)$,
even when $i\notin A$,
so $\mu^{i+1}$ is also defined.

We prove (\ref{17.eq.7})--(\ref{w0}) by induction on $i$.
If $i=0$ then, as
$\mu\unrhd\lambda(a)$, (\ref{17.eq.7}) holds. Also,
(\ref{w1}) holds since $r\in R(\mu,a,r)$.
Since $0\notin A$, we have $l_1=l(\mu,a,l_0)=r\leq s(\mu,a)$.
Thus (\ref{w0}) holds for $i=0$ as well.

Assume (\ref{17.eq.7})--(\ref{w0}) hold for some $i\in\{0,1,\dots,n-2\}$.
Since $l_{i+1}\in R(\mu^i,a^i)$,
(\ref{17.eq.7}) holds for $i+1$.

If $i+1\notin A$, then
$a^i_{h(a^i)}\geq2$. Also $l_{i+1}\leq s(\mu^i,a^i)$ by induction.
Lemma~\ref{lem.6} then implies $R(\mu^{i+1},a^{i+1},l_{i+1})\neq\emptyset$
and $l_{i+2}\leq s(\mu^{i+1},a^{i+1})$,
so (\ref{w1})--(\ref{w0}) hold for $i+1$ as well.

If $i+1\in A$, then
$l_{i+2}=l(\mu^{i+1},a^{i+1},1)=\min R(\mu^{i+1},a^{i+1})
\leq s(\mu^{i+1},a^{i+1})$ by Lemma~\ref{lem.10}.
Thus (\ref{w0}) holds for $i+1$ as well.

Clearly, $\mu^{i}=\mu^{(l_1,\ldots,l_{i})}$.
\end{nota}

\begin{lem}\label{lem.3f}
Let $T\in\STab(\mu,a)$.
With reference to Notation~\ref{not.3},
suppose
$T^{-1}(h)=\{(t_1,t'_1),(t_2,t'_2),\ldots,(t_{a_h},t'_{a_h})\}$ and
$r\leq t_1\leq t_2\leq\cdots\leq t_{a_h}$.
Then  $l_i\leq t_i$ for $1\leq i\leq a_h$.
In particular, $\mu^{(t_1,\ldots,t_{a_h})}\unrhd\mu^{a_h}$.
\end{lem}
\begin{proof}
We prove the assertion by
induction on $i$.
If $i=1$, then $l_1=r\leq t_1$.

Assume $l_1\leq t_1,\dots,l_i\leq t_i$ hold
for some $i$ with $1\leq i<a_h$.
We aim to show $l_{i+1}\leq t_{i+1}$ by deriving
\begin{equation}\label{eq:lem.3fA}
R(\mu^{(t_1,\dots,t_i)},a^i,t_i)\subseteq
R(\mu^i,a^i,l_i)
\end{equation}
from Lemma~\ref{lem.3d}. In order to do so, we need to verify
the hypotheses of Lemma~\ref{lem.3d}.
By the definition of $l_i$, we
have $l_i\in R(\mu^{i-1},a^{i-1})$.
Since the restriction of $T$ to $D_{\mu^{(t_1,\dots,t_{i-1})}}$
is an element of $\STab(\mu^{(t_1,\dots,t_{i-1})},a^{i-1})$,
Lemma~\ref{lem.4b} implies
$t_i\in R(\mu^{(t_1,\dots,t_{i-1})},a^{i-1})$,
and our inductive hypothesis shows $l_i\leq t_i$.
Similarly, we have
\begin{equation}\label{eq:lem.3f}
t_{i+1}\in R(\mu^{(t_1,\dots,t_{i})},a^{i},t_i).
\end{equation}
Since $l_p\leq t_p\leq t_i$ for $1\leq p\leq i-1$
by our inductive hypothesis,
\begin{align*}
\mu^{i-1}_{t_i}&=\mu^{(l_1,\dots,l_{i-1})}_{t_i}
\nexteq
\mu_{t_i}-|\{p\mid 1\leq p\leq i-1,\;l_p=t_i\}|
\nextgeq
\mu_{t_i}-|\{p\mid 1\leq p\leq i-1,\;t_p=t_i\}|
\nexteq
\mu^{(t_1,\dots,t_{i-1})}_{t_i}.
\end{align*}
Finally, for $j>t_i$, we have
$\mu^{i-1}_{j}=\mu^{(l_1,\dots,l_{i-1})}_{j}=
\mu_j=\mu^{(t_1,\dots,t_{i-1})}_{j}$,
since $l_p\leq t_p\leq t_i$ for $1\leq p\leq i-1$.
Therefore, we have verified all the hypotheses of
Lemma~\ref{lem.3d}, and we obtain (\ref{eq:lem.3fA}).

Now
\begin{align*}
l_{i+1}&=l(\mu^i,a^i,l_i)
\nexteq
\min R(\mu^i,a^i,l_i)
\nextleq
\min R(\mu^{(t_1,\dots,t_i)},a^i,t_i)
&&\text{(by (\ref{eq:lem.3fA}))}
\nextleq
t_{i+1}
&&\text{(by (\ref{eq:lem.3f})).}
\end{align*}
\end{proof}

\begin{thm}\label{thm.3}
Let $\mu=(\mu_1,\mu_2,\ldots,\mu_k)\vdash
n$ and $a=(a_1,a_2,\ldots,a_h)\vDash n$, and suppose
$\mu\unrhd\lambda(a)$.
Let $r\in R(\mu,a)$ and suppose $r\leq s(\mu,a)$.
Define $a^i$, $l_i$ and $\mu^i$ as in Notation~\ref{not.3}.
Define a tableau $S$ of shape
$\mu$ and weight $a$ by
\begin{equation*}
S(l_{i+1},\mu^i_{l_{i+1}})=t,
\end{equation*}
where $0\leq i<n$ and
$\sum_{j=1}^{t-1} a_j< n-i\leq\sum_{j=1}^t a_j$. Then $S$ is
the least element of the subposet
\begin{equation}\label{eq:thm.3A}
\{T\in\STab(\mu,a)\mid\min\{i\mid T(i,\mu_i)=h\}\geq r\}
\end{equation}
In particular, if $r=l(\mu,a)$, then $S$ is the least element of
$\STab(\mu,a)$.
\end{thm}

\begin{proof}
Note that the tableau $S$ is well-defined. Indeed, by the definition of
$\mu^{i}$, we have
\begin{equation}\label{17.eq.13}
D_{\mu^{i}}=D_{\mu^{i+1}}\cup\{(l_{i+1},\mu^{i}_{l_{i+1}})\}.
\end{equation}
So $D_{\mu}=\{(l_{i+1},\mu^i_{l_{i+1}})\mid 0\leq i<n\}$.

Next, we prove the statement by induction on $n$. If $n=1$ then
$\mu=a=(1)$, so it is obvious.

Assume that the statement holds for $n-1$. We apply Notation~\ref{not.3}
with $r,\mu,a$ replaced by $l_2,\nu=\mu^1,b=a^1$, respectively.
This is admissible since
 $l_2\in R(\mu^1,a^1)=R(\nu,b)$ and $l_2\leq s(\mu^1,a^1)=s(\nu,b)$ by (\ref{w0}).
Define $b^i$, $l'_i$ and $\nu^i$ inductively by setting
$b^0=b$, $l'_0=l_2$, $\nu^0=\nu$
and for $0\leq i< n-1$,
\begin{align*}
b^{i+1}&=\widetilde{b^{i}}\vDash n-i-2,\\
l'_{i+1}&=
\begin{cases}
l(\nu^i,b^i,1)&\text{ if } i\in B,\\
l(\nu^i,b^i,l'_{i})&\text{ otherwise, }
\end{cases}\\
\nu^{i+1}&=(\nu^{i})^{(l'_{i+1})}\vdash n-i-2,
\end{align*}
where
\begin{align*}
B&=\begin{cases}
\{b_{h-1},b_{h-1}+b_{h-2},\dots,b_{h-1}+\cdots+b_2\}
&\text{if $a_h=1$,}\\
\{b_{h},b_{h}+b_{h-1},\dots,b_{h}+\cdots+b_2\}
&\text{otherwise,}
\end{cases}\\
b&=\begin{cases}
(b_1,\dots,b_{h-1})&\text{if $a_h=1$,}\\
(b_1,\dots,b_{h})&\text{otherwise.}
\end{cases}
\end{align*}
Define a tableau
$\tilde{S}$ of shape $\nu$ and weight $b$ by
\begin{equation*}
\tilde{S}(l'_{i+1},\nu^i_{l'_{i+1}})=t,
\end{equation*}
where
$\sum_{j=1}^{t-1} b_j< n-1-i\leq\sum_{j=1}^t b_j$. By the
inductive hypothesis, $\tilde{S}$ is the least element of the set
\begin{equation}\label{eq:Tmin}
\{\tilde{T}\in\STab(\nu,b)\mid
\min\{i\mid \tilde{T}(i,\nu_i)=h(b)\}\geq l_2\}.
\end{equation}

It is easy to see that $b^i=a^{i+1}$ for $0\leq i<n$. We
show that
\begin{equation}\label{eq:ll}
l'_{i}=l_{i+1}\text{ and }\nu^i=\mu^{i+1}\quad(1\leq i<n)
\end{equation}
by induction on $i$.
Since $0\notin B$, we have
\begin{align}\label{eq:thm3.1}
l'_1&=l(\nu^0,b^0,l'_0)\notag
\nexteq
l(\mu^1,a^1,l_2)\notag
\nexteq
\begin{cases}
l(\mu^1,a^1,l(\mu^1,a^1,1))&\text{if $1\in A$,}\\
l(\mu^1,a^1,l(\mu^1,a^1,l_1))&\text{otherwise}
\end{cases}\notag
\nexteq
\begin{cases}
l(\mu^1,a^1,1)&\text{if $1\in A$,}\\
l(\mu^1,a^1,l_1)&\text{otherwise}
\end{cases}\notag
\nexteq
l_2.
\end{align}
Then $\nu^1=(\nu^0)^{(l'_1)}=(\mu^1)^{(l_2)}=\mu^2$.

Assume $i\geq2$ and $l'_{i-1}=l_{i}$ and $\nu^{i-1}=\mu^{i}$. Since $i-1\in B$ if and only if $i\in A$, we have
\begin{align*}
l'_{i}&=
\begin{cases}
l(\nu^{i-1},b^{i-1},1)&\text{ if } {i-1}\in B,\\
l(\nu^{i-1},b^{i-1},l'_{i-1})&\text{ otherwise }
\end{cases}\nexteq
\begin{cases}
l(\mu^{i},a^{i},1)&\text{ if } i\in A,\\
l(\mu^{i},a^{i},l_{i})&\text{ otherwise }
\end{cases}\nexteq
l_{i+1}.
\end{align*}
Then
$\nu^{i}=(\nu^{i-1})^{(l'_{i})}=(\mu^{i})^{(l_{i+1})}=\mu^{i+1}$.

Next we show
\begin{equation}\label{eq:StS}
S|_{D_{\nu}}=\tilde{S}.
\end{equation}
First, since $b=a^1$, we obtain
\[
\sum_{i=1}^jb_i=\begin{cases}
\sum_{i=1}^ja_i&\text{ if }j<h,\\
\sum_{i=1}^ha_i-1&\text{ if }j=h.
\end{cases}
\]
Suppose that $\sum_{j=1}^{t-1} b_j< n-1-i\leq\sum_{j=1}^{t} b_j$.
Then
\[
\sum_{j=1}^{t-1}a_j<n-(i+1)\leq\sum_{j=1}^ta_j,
\]
so $\tilde{S}(l'_{i+1},\nu^i_{l'_{i+1}})=t=S(l_{i+2},\mu^{i+1}_{l_{i+2}})$.
Thus, we have proved (\ref{eq:StS}).

Next we show $S\in\STab(\mu,a)$.
If $l_1=1$ then this is clear, since
$\tilde{S}\in\STab(\nu,b)$. Suppose $l_1\geq 2$. Since $(l_1-1,\mu_{l_1})\in
D_{\nu}$, there exists an $i\in\{1,2,\dots,n-1\}$ such that
$(l_1-1,\mu_{l_1})=(l_{i+1},\mu^i_{l_{i+1}})$.
Since
$l_1\leq l_2\leq\cdots\leq l_{a_h}$, we have
\[i+1> a_h=n-\sum_{j=1}^{h-1}a_j=n-\sum_{j=1}^{h-1}b_j,\]
and hence
\[n-1-(i-1)\leq\sum_{j=1}^{h-1}b_j.\]
This implies
\begin{equation}\label{17.eq.2}
\tilde{S}(l'_{i},\nu^{i-1}_{l'_{i}})\leq h-1.
\end{equation}
Now
\begin{align*}
S(l_1-1,\mu_{l_1})&=
S(l_{i+1},\mu^i_{l_{i+1}})
\nexteq
\tilde{S}(l'_{i},\nu^{i-1}_{l'_{i}})
&&\text{(by (\ref{eq:ll}), (\ref{eq:StS}))}
\nextlq h
&&\text{(by (\ref{17.eq.2}))}
\nexteq S(l_1,\mu_{l_1}).
\end{align*}
Since $\tilde{S}\in\STab(\nu,b)$, this implies
$S\in\STab(\mu,a)$.

It remains to show that $S\leq T$ for all $T$ in the set (\ref{eq:thm.3A}).
Define partitions $\tau$ and $\sigma$ by (\ref{Tt}) and (\ref{Ss}),
respectively.

Suppose first that $\min\{i\mid T(i,\mu_i)=h\}>l_1$. Write
$T^{-1}(h)=\{(t_1,t'_1),\ldots,(t_{a_h},t'_{a_h})\}$ with
$l_1<t_1\leq t_2\leq\cdots\leq t_{a_h}$.
Then Lemma~\ref{lem.3f} implies
$\tau=\mu^{(t_1,\ldots,t_{a_h})}\unrhd
\mu^{a_h}=\sigma$.
Since $l_1<t_1$, we have $\tau\rhd\sigma$.
Thus $S\leq T$.

Next suppose that $\min\{i\mid T(i,\mu_i)=h\}=l_1$.
Set $\tilde{T}=T|_{D_{\nu}}$ and observe
$\tilde{T}\in\STab(\nu,b)$.
Set $m=\min\{i\mid \tilde{T}(i,\nu_i)=h(b)\}$.
By Lemma~\ref{lem.4b},
we have $m\in R(\nu,b)$, so $m\geq l(\nu,b)$.
If $a_h=1$, then $l_2=l(\nu,b)$, so $m\geq l_2$. If $a_h\geq2$, then
$h(b)=h$, so $m\geq l_1$. Thus $m\geq l(\nu,b,l_1)=l_2$.
Therefore,
$\tilde{T}$ belong to the set (\ref{eq:Tmin}). This implies
$\tilde{S}\leq\tilde{T}$, and hence
either
$\tau\rhd\sigma$, or
$\tau=\sigma$ and $\tilde{S}|_{D_\sigma}\leq \tilde{T}|_{D_\sigma}$.
Since $\tilde{S}|_{D_\sigma}=S|_{D_\sigma}$ and
$\tilde{T}|_{D_\tau}=T|_{D_\tau}$, the recursive definition of
the partial order implies $S\leq T$.
\end{proof}

\begin{alg}\label{a.1}
\textbf{Input}: $\mu=(\mu_1,\mu_2,\ldots,\mu_k)\vdash n$ and
$a=(a_1,a_2,\ldots,a_h)\vDash n$ such that $\mu\unrhd\lambda(a)$.\\
\textbf{Output}: $S\in\STab(\mu,a)$.\\
\textbf{Initialization}: $\nu:=\mu$, $b:=a$, $m:=n$ and $l':=1$.\\
\textbf{while} $m>1$ \textbf{do}\\
\begin{tabular}{c|l}
&\;$h:=h(b)$ and $l:=l(\nu,b,l')$.\\
&\;$S(l,\nu_l):=h$.\\
&\;\textbf{if} $b_h=1$, \textbf{then}
$l'\leftarrow 1$.\\
&\;\textbf{else} $l'\leftarrow l$.\\
&\;$\nu\leftarrow\nu^{(l)}$, $b\leftarrow\tilde{b}$ and $m\leftarrow
m-1$.\\
\end{tabular}\\
\textbf{end}\\
$S(1,1):=1$.\\
Output $S$.
\end{alg}

\begin{exam}
Let $\mu=(4,4,1,1)\vdash 10$ and $a=(1,3,2,2,2)\vDash 10$. Then a
tableau $S\in\STab(\mu,a)$ is obtained via Algorithm~\ref{a.1}.
\[
\begin{tabular}{|c|c|c|c|c|c|c|}
\hline
$\nu$ & $b$ & $m$ & $l'$ & $h$ & $l$ & $S$  \\
\hline $(4,4,1,1)$ & $(1,3,2,2,2)$ & $10$ & $1$ & $5$ & $2$ & $S(2,4)=5$ \\
\hline $(4,3,1,1)$ & $(1,3,2,2,1)$ & $9$ & $2$ & $5$ & $2$ & $S(2,3)=5$ \\
\hline $(4,2,1,1)$ & $(1,3,2,2)$ & $8$ & $1$ & $4$ & $1$ & $S(1,4)=4$ \\
\hline $(3,2,1,1)$ & $(1,3,2,1)$ & $7$ & $1$ & $4$ & $4$ & $S(4,1)=4$ \\
\hline $(3,2,1)$ & $(1,3,2)$ & $6$ & $1$ & $3$ & $2$ & $S(2,2)=3$ \\
\hline $(3,1,1)$ & $(1,3,1)$ & $5$ & $2$ & $3$ & $3$ & $S(3,1)=3$ \\
\hline $(3,1)$ & $(1,3)$ & $4$ & $1$ & $2$ & $1$ & $S(1,3)=2$ \\
\hline $(2,1)$ & $(1,2)$ & $3$ & $1$ & $2$ & $1$ & $S(1,2)=2$ \\
\hline $(1,1)$ & $(1,1)$ & $2$ & $1$ & $2$ & $2$ & $S(2,1)=2$ \\
\hline $(1)$ & $(1)$ & $1$ & $1$ &  &  & $S(1,1)=1$ \\
\hline
\end{tabular}
\]
Thus
\[
S=\begin{array}{cccc}
1&2&2&4\\
2&3&5&5\\
3&&&\\
4
\end{array}
\]
and $S$ is the least element of $\STab(\mu,a)$.
\end{exam}

\begin{rmk}
Let $a$ and $b$ be compositions of $n$ with $\lambda(a)=\lambda(b)$. Then
there exists a bijection from $\STab(\mu,a)$ to $\STab(\mu,b)$ using \cite[Lemma 3.7.1]{C}, 
but they are not isomorphic as partially ordered sets.
For example, let $\mu=(4,4,1,1)\vdash 10$, $a=(1,3,2,2,2)\vDash 10$ and $b=(1,2,2,2,3)\vDash 10$. 
Then $(\STab(\mu,b),\leq)$ is a totally ordered set, while 
$(\STab(\mu,a),\leq)$ contains two incomparable tableaux:
\[
T=\begin{array}{cccc}
1&2&2&3\\
2&4&4&5\\
3&&&\\
5
\end{array},
\]
\[
S=\begin{array}{cccc}
1&2&2&4\\
2&3&3&5\\
4&&&\\
5
\end{array}.
\]
Indeed, define $\tau^{(p)}$ and $\sigma^{(p)}$ by (\ref{def.eq.1}) and (\ref{def.eq.2}), respectively. 
Then $\tau^{(3)}=(4,1,1)$ and $\sigma^{(3)}=(3,3)$ are incomparable. Thus
$(\STab(\mu,a),\leq)$ is not totally ordered, hence it is not isomorphic to $(\STab(\mu,b),\leq)$.
\end{rmk}

\end{document}